\documentclass[11pt, reqno]{amsart}
\usepackage{amsmath,amssymb,amsfonts,amscd,hyperref,color}
\usepackage[utf8]{inputenc}
\usepackage{csquotes}
\usepackage{nicefrac}

\usepackage[abs]{overpic}
\usepackage{verbatim}
\usepackage{enumerate}

\usepackage[
backend=biber,
style=alphabetic,
sorting=nyt,
maxnames = 99,
minnames = 99,
maxalphanames=99,
maxcitenames=99,
maxbibnames = 99,
doi=false,
url=false,
isbn= false
]{biblatex}

\AtEveryBibitem{%
	\clearfield{issn}%
}

\newcommand{\Hmm}[1]{\leavevmode{\marginpar{\tiny%
$\hbox to 0mm{\hspace*{-0.5mm}$\leftarrow$\hss}%
\vcenter{\vrule depth 0.1mm height 0.1mm width \the\marginparwidth}%
\hbox to
0mm{\hss$\rightarrow$\hspace*{-0.5mm}}$\\\relax\raggedright #1}}}

\newtheorem{theorem}{Theorem}
\newtheorem{corollary}[theorem]{Corollary}
\newtheorem{lemma}[theorem]{Lemma}
\newtheorem{proposition}[theorem]{Proposition}

\theoremstyle{definition}

\newtheorem*{remark}{Remark}

\numberwithin{equation}{section}
\newcommand{\Z}{{\mathbb Z}}
\newcommand{\R}{{\mathbb R}}

\newcommand{\N}{{\mathbb N}}

\newcommand{\D}{{\mathbb D}}

\let\L\undefined

\newcommand{\L}{\mathcal{L}}

\newcommand{\F}{\mathcal{F}}
\renewcommand{\P}{\mathcal{P}}

\newcommand{\abs}[1]{\left| #1 \right|}
\newcommand{\norm}[1]{\left\| #1 \right\|}

\renewcommand{\L}{\mathcal{L}}
\renewcommand{\D}{\mathcal{D}}

\begin{document}
\title[Supersolution Construction and Optimal Hardy Inequality]
{Supersolution Construction and Optimal Hardy Inequality for Fractional Laplacians}

\author[P.~Hake]{Philipp Hake}
\address{P.~Hake, Institut f\"ur Mathematik, Universit\"at Leipzig, 04109 Leipzig, Germany}\email{philipp.hake@math.uni-leipzig.de}
\author[M.~Keller]{Matthias Keller}
\address{M.~Keller, Israel Institute of Advanced Studies, Jerusalem, Israel; Institut f\"ur Mathematik, Universit\"at Potsdam
14476  Potsdam, Germany}
\email{matthias.keller@uni-potsdam.de}
\author[F.~Pogorzelski]{Felix Pogorzelski}
\address{ F.~Pogorzelski, Israel Institute of Advanced Studies, Jerusalem, Israel;  Institut f\"ur Mathematik, Universit\"at Leipzig, 04109 Leipzig, Germany}
\email{felix.pogorzelski@math.uni-leipzig.de}
 
\date{\today}

\begin{abstract}
We give a new criterion to show optimality of Hardy weights for general operators on graphs via the supersolution construction. 
For Laplacians on graphs without killing terms this always gives rise to an optimal Hardy weight via the Green's function without any further assumptions. Furthermore, in contrast to earlier results, our result is not restricted to locally finite graphs. This allows us in particular to obtain optimal Hardy weights for the fractional Laplacian on general graphs. For the fractional Laplacian on the Euclidean lattice, we then obtain an optimal Hardy weight with the constant and asymptotics as it is expected from the continuous setting. 
\end{abstract}
\maketitle

\tableofcontents

\section{Introduction}
The Hardy inequality  goes back to G.H.~Hardy's  desire for a simple proof of Hilbert's double series theorem, \cite{HardyHistory} and has since then become a classical result in analysis. Its continuous version is a fundamental tool in the analysis of partial differential equations, mathematical physics and spectral theory as it serves for example as a quantitative version of the uncertainty principle. 
The most prominent mathematical features of the Hardy inequality are that it allows to determine the optimal constant and to show the absence of a minimizer 
in many specific cases.

To move beyond these specific cases, a major breakthrough was the discovery of the supersolution construction in \cite{DFP14} answering a question of Agmon \cite{Agmon}. This method allows for the construction of optimal Hardy weights from positive superharmonic functions for general second-order differential operators. Here the notion of optimality consists of three properties that go beyond  optimality of the  constant. The first property is that the Hardy weight is critical, which means that it cannot be increased. Secondly, the Hardy weight is null-critical, which means that the Hardy inequality does not admit a minimizer. Thirdly, the Hardy weight is optimal near infinity, which means that outside of finite sets, it cannot be increased by a factor greater than one. This in particular implies the optimality of the constant, given the asymptotic behavior of the Hardy weight at infinity is known. Later it was shown that the third property actually follows  from the first two properties, see also \cite{KovarikPinchover,KPP18,Fischer,KN,HKPII}.

Going back to the roots of Hardy's original inequality in the discrete setting, the supersolution construction was transferred to  Schrödinger operators on graphs in \cite{KPP18}. Surprisingly this led to an improvement of Hardy's original inequality as it yields positive higher order terms in the Hardy weight \cite{KPPHardy} and also allowed to recover the expected constant $(d-2)^2/4$ at infinity in $\Z^d$, $d\ge 3$ which was not known before, cf.~\cite{RS09,KL16} and \cite{Gup23} as well. Furthermore, this method was used for trees \cite{BSV}, more general elliptic operators on $\Z^d$ \cite{KL23} and also for the quasilinear $p$-Laplacian \cite{Fischer,FKP23}.

However, the supersolution construction introduced in \cite{KPP20} has a substantial shortcoming as its assumptions (a bounded oscillation condition and properness of the harmonic function) exclude non-locally finite graphs and, in particular,  fractional Laplacians. The first discrete Hardy inequality of the fractional Laplacian on $\Z^d$ was already obtained in \cite{CR18} but even  optimality  of the constant was unknown. A bit later, optimality for this weight was  shown in \cite{DFF,KN}  for dimension one. Recently, this result was extended to $\Z^d$ in \cite{HKP} and more general graphs satisfying certain heat kernel bounds in \cite{HKPII}.
These results use a very different  method of proof and need substantial assumptions on the underlying graph.

In this paper we give a new criterion for obtaining optimal Hardy weights via the supersolution construction, Theorem~\ref{thm:null_critical_weights_Hardy}.  This criterion in particular does not entail any restriction such as local finiteness.
The core idea is to  use a positive potential for the supersolution construction. Since potentials can be approximated well enough by finitely supported functions, this approximation then morally plays the role of the null-sequence in the proof of criticality, cf.\@ \cite[Theorem~5.3]{KPP18}. 

Hence,  the Green's function is a natural candidate and always gives a critical Hardy weight, see Theorem~\ref{thm:optGF}. Moreover, it is null-critical in the case when there is no potential term. In the  presence of a positive potential term, our criterion for null-criticality can be characterized in terms of stochastic completeness properties of the graph, see \cite{HKPIV} and  the remark following Theorem~\ref{thm:optGF}.

The latter result is then applied to the fractional Laplacian, where the Green's function is given by the Riesz kernel, to obtain optimal Hardy weights for  general graphs, see Theorem~\ref{thm:FrLa}. We finally compare the optimal Hardy weight obtained via the supersolution construction with the optimal Hardy weight obtained in \cite{HKPII}
on $\Z^d$. It turns out that both weights share the same constant and  asymptotics at infinity.

The abstract main result, Theorem~\ref{thm:null_critical_weights_Hardy}, is stated for graphs with positive or vanishing potentials. Later, we show in Theorem~\ref{thm:main_optimality_theorem_Schr} how this can be easily extended to positive Schr\"odinger operators using the ground state transform.

\medskip

{\bf Organization of the article}. In the next section we give the basic definitions and state the main results which were discussed above. In Section~\ref{sec:prelim} we prove the main technical lemma which is essential to prove a result for so called Hardy-type inequalities, i.e., inequalities where the weights are not necessarily non-negative in Section~\ref{sec:null_critical_weights}. This result is then used in the proofs of Theorem~\ref{thm:null_critical_weights_Hardy} and Theorem~\ref{thm:optGF}. 
 Afterwards, in Section~\ref{sec:fractional} we discuss the fractional Laplacian and prove Theorems~\ref{thm:FrLa} and~\ref{thm:Zd}. Finally, in Section~\ref{sec:Schr} we extend  Theorem~\ref{thm:null_critical_weights_Hardy} to positive Schr\"odinger operators.\medskip

We emphasize that the abstract main results, Theorem~\ref{thm:null_critical_weights_Hardy} and Theorem~\ref{thm:main_optimality_theorem} are extensions of  \cite[Theorems~2.23 and~2.24]{hake2025optimal} from the first named author's doctoral dissertation thesis to graphs with a killing term over an arbitrary discrete measure space.

\section{Main results} 

Let $X$ be countable, discrete set. A {\em graph} $(b,c)$ with vertices in $X$ is a pair consisting of a symmetric function $b:X \times X \to [0, \infty)$ with $b(x,x) = 0$ and $\deg(x) := \sum_{y \in X} b(x,y) < \infty$ for all $x \in X$, and a non-negative map $c:X \to [0,\infty)$. The map $c$ is called the {\em killing term} of the graph. If $c=0$, then we simply write $b$ for the graph $(b,0)$. Endowing $X$ with a function $m:X \to (0,\infty)$ canonically gives rise to a measure of full support, and thus, to a discrete measure space $(X,m)$, and we say that $(b,c)$ is a {\em graph over $(X,m)$}. Furthermore, we associate the canonical $\ell^p(X,m)$ spaces with norms ${\|\cdot\|}_p$ to $(X,m)$ for $p \in [1,\infty]$ in the usual way.
See \cite{KLW} for an in-depth discussion of the theory of graphs over discrete measure spaces.

We will assume throughout this work that all graphs $(b,c)$ considered are {\em connected}, i.e.,\@ for every  $x,y \in X$ there are $x= y_1,y_2, \dots, y_n, y_{n+1}=y$ such that $b(y_i, y_{i+1})>0$ for all $1 \leq i \leq n$. A graph $(b,c)$ over $(X,m)$ gives rise to the {\em formal Laplacian}
 $\L:\mathcal{F} \to C(X)$ defined on the space of functions $$\mathcal{F}=\{f:X\to \R\mid \sum_{y\in X} b(x,y)|f(y)|<\infty\mbox{ for all }x\in X\},$$
where we denote by $C(X)$ the vector space of real-valued functions on $X$.
The formal Laplacian acts as
\[\L f(x)=\frac{1}{m(x)}\sum_{y\in X}b(x,y)(f(x)-f(y)) \,+ \, \frac{c(x)}{m(x)} f(x).\] 
The associated quadratic form $\mathcal{Q}$ is defined on the space of {\em functions of finite energy} \[\mathcal{D}=\{f:X\to \R \mid \mathcal{Q}(f)<\infty\},\]
and acts for $f \in \mathcal{D}$ as 
\[\mathcal{Q}(f)=\frac{1}{2}\sum_{x,y\in X}b(x,y)(f(x)-f(y))^2+\sum_{x\in X}c(x)f(x)^2.\] 
We denote by $C_c(X) \subseteq C(X)$  the subspace of all finitely supported functions and by $C_0(X) $ its closure with respect to the uniform norm ${\|\cdot\|}_\infty$. 
The {\em extended space} $$\mathcal{D}_0:=\mathcal{D}_0(\mathcal{Q}):=\overline{C_c(X)}^{\|\cdot\|_{o}}$$ is defined as the closure of $C_c(X)$ in $\mathcal{D}$ with respect to the norm $\|f\|_{o}=\sqrt{\mathcal{Q}(f)+|f(o)|^2}$, where $o \in X$ is some fixed vertex. The restriction of $\mathcal{Q}$ to $\mathcal{D}_0\cap \ell^2(X,m)$ gives rise to a regular Dirichlet form, see \cite[Theorem~1.19]{KLW}, and we denote by $L$ the associated self-adjoint operator on $\ell^2(X,m)$.

For a summable or positive function $f$ over a discrete set $A$ (typically a subset of $X$ or $X\times X$), we write $$\sum_{X} f=\sum_{x \in X} f(x).$$
Given $x \in X$, we write $1_x = 1_{\{x\}} \in C_c(X)$ for the function taking the value 1 at $x$ and 0 elsewhere.
\medskip

A   function $w:X \to [0,\infty)$ is   a \emph{Hardy  weight for $(b,c)$ over $(X,m)$} if
\[\mathcal{Q}(\varphi)\ge \sum_{ X}mw\varphi^2\]
for all $\varphi\in C_c(X)$. Whenever the inequality holds, it can be extended to all $\varphi\in \mathcal{D}_0$.
The graph $(b,c)$ is called \emph{transient} if  there exists a Hardy weight $w \neq 0$, cf.~\cite{F00,KLW} which is clearly the case if $c\neq0$. It is well-known \cite[Theorem~6.1 and 6.29]{KLW} that a graph  is transient if and only if for every $y \in X$,  there is a unique $G_y \in \D_0$ such that $\mathcal{L} G_y = 1_y$. The function $$G\colon X \times X \to (0,\infty), \quad G(x,y) = G_y(x)$$ is called the {\em Green's function} of $(b,c)$ associated with $\mathcal{L}$. 
The Green's operator is then defined as $G\colon\mathcal{G} \to C(X)$, where 
\begin{align*}
 	\mathcal{G} =\{k\colon X\to \R \mid\sum_{y\in X}G(x,y)|k(y)|< \infty \mbox{ for all } x \in X \}, \end{align*}
	and
	\begin{align*}
 		Gk(x)=\sum_{y\in X}G(x,y)k(y), 
\end{align*}
for $ k \in \mathcal{G}$  and   $x \in X$.
Every $u \in \mathcal{F}$ arising as $u=Gk$ with $k \in \mathcal{G}$ is called a {\em potential}. We write $$\mathcal{P}=\{u\in \mathcal{F} \mid u=Gk \text{ for some } k \in \mathcal{G}\}$$ for the set of all potentials. It can be shown, see e.g. \cite[Lemma~4]{HKPIV}, that $c\in \mathcal{G}$ and $Gc\leq 1$.
\medskip

A Hardy  weight $w$ is {\em critical} if for every Hardy weight $w^{\prime} \geq w$, we have   $w^{\prime} = w$. It is well known  \cite[Theorem~5.3]{KPP20} that criticality of $w$ is equivalent to the existence of a unique (up to multiplication by scalars) strictly positive $v \in \mathcal{F}$ such that $\mathcal{L}v = wv$. Such a function $v \in \mathcal{F}$ is called an {\em Agmon ground state} for $w$. 
A critical Hardy  weight $w$ is said to be {\em positive critical} if $v$ is square integrable with respect to the measure $wm$ (which might not have full support), i.e., $v \in \ell^2(X,wm)$,	 which is equivalent to the Hardy inequality admitting a minimizer, \cite[Theorem~6.2]{KPP20}. A critical Hardy weight is {\em null-critical} if it is not positive critical. Null-critical Hardy weights are also {\em optimal near infinity}, which means that outside of finite sets, they cannot be increased by a factor greater than one, 
 cf.\@ \cite[Proposition~15]{HKPII}. Correspondingly, null-critical Hardy weights are also referred to as {\em optimal Hardy weights}. 

\medskip

Our first main result is a new criterion for obtaining null-critical Hardy weights over general graphs. 

\begin{theorem}[Null-critical Hardy weights] \label{thm:null_critical_weights_Hardy}
	Let  $(b,c)$ be a transient connected graph over $(X,m)$. Then, for all 	strictly positive, superharmonic
	 $u \in \mathcal{P}$ with $\mathcal{L}u \in \ell^1(X,m)$, the function
	 $w = {\mathcal{L} u^{\frac{1}{2}}}/{u^{\frac{1}{2}}}$ is a  critical  Hardy weight.
If furthermore,  $G(c/m)(x) <1$ for some $x \in X$, then each such $w$ is null-critical (and, hence, optimal).
\end{theorem}

 \begin{remark} [The supersolution construction] Considering weights of the form $\mathcal{L}u^{1/2}/u^{1/2}$ for some strictly positive and superharmonic function is called the {\em supersolution construction}. It was established first in the continuum setting by Devyver, Fraas and Pinchover in \cite{DFP14} and later transferred to an abstract class of graphs by two of the authors with Pinchover in \cite{KPP18}. In order to obtain a null-critical Hardy weight, \cite[Theorem~3.1]{KPP18} requires $u$ to be harmonic outside a finite set, to be proper (i.e.,\@ $u^{-1}(K)$ is finite for each compact set $K \subseteq (0,\infty)$), and to satisfy a bounded oscillation condition on values on the edges. Clearly, $\mathcal{L}u \in \ell^1(X,m)$ is a substantially weaker condition than   harmonicity outside finite sets. Moreover, the formerly assumed properness and bounded oscillation condition  exclude non-locally finite graphs which we no longer require.
Yet, these criteria are complemental to the assumption $u \in  \mathcal{P}$ in the above theorem. The rather subtle relations are explained for example in detail in \cite[Remark~2.27, Theorem~2.28]{hake2025optimal}. We also point out that neither of the assumptions is necessary to obtain a null-critical and hence optimal Hardy weight, see \cite{FR25} for non-necessity of the assumption $\mathcal{L}u \in \ell^1(X,m)$
and \cite{HKPV} for $u \in  \mathcal{P}$.
\end{remark}

\begin{remark}[The condition $G(c/m)<1$] The condition $G(c/m)<1$ is trivially fulfilled for $c=0$. In a recent paper, \cite{HKPIV}, the condition and $G(c/m) = 1$ was characterized in terms of a Liouville theorem for bounded harmonic functions. Also, $G(c/m)(x) < 1$ for some $x \in X$ if and only if this holds for all $x \in X$, cf.\@ \cite[Theorem~5 or Theorem~26]{HKPIV}.  Furthermore, this criterion is related to stochastic completeness properties and  validity of a Green's formula for the graph under consideration. We will explore this in more detail in the theorems below.
\end{remark}

\begin{remark}[Superharmonic $u\in \mathcal{D}_0$]
	We point out that all superharmonic $u \in  \mathcal{D}_0$  are  necessarily  potentials, cf.\@ \cite[Lemma~8]{HKPII}. 
\end{remark}

Our next result now specializes on the case of choosing $u$ as a Green's function. 
The advantage is that the Green's function automatically satisfies all conditions  except for the last one in case of $c\neq 0$. 
However, in this case we can characterize this condition in terms of stochastic completeness properties of graphs. 
We only state what is needed here and refer to \cite[Chapter~7]{KLW} for in-depth background information on the concept of stochastic completeness. 
A connected graph $(b,c)$ over $(X,m)$ is called {\em stochastically complete} if $e^{-tL}1 = 1$ for all  $t \geq 0$, where $e^{-tL}$ denotes the semigroup associated with the Dirichlet-Laplacian $L$, see \cite[Chapter~1]{KLW}. 
Any graph with non-vanishing killing term $c \neq 0$ is necessarily {\em stochastically incomplete}, i.e.\@ not stochastically complete, and we call $(b,c)$  {\em stochastically complete at infinity} if for some (equivalently all) $t >0$ and for some (equivalently all) $o \in X$, we have 
\begin{align*}
	M_t(o) := e^{-tL}1(o) \, + \, \int_0^{t} \Big( e^{-sL} \frac{c}{m}  \Big)(o) \,ds=1.
\end{align*}
The underlying idea of stochastic completeness is that all the heat is conserved within the graph. On the other hand, a killing term constantly removes heat from the graph so it becomes stochastically incomplete. Stochastic completeness at infinity means that the heat is conserved within the graph if we also take into account the heat removed by the killing term.
Clearly, every stochastically complete graph is also stochastically complete at infinity.

\begin{theorem}[Supersolution construction for the Green's function] \label{thm:optGF}
	Let $(b,c)$ be a transient connected graph over $(X,m)$ and fix $o \in X$. 
	Then 
	\begin{align*}
		{w} = \frac{\mathcal{L}G^{1/2}_o}{G^{1/2}_o}
	\end{align*}
	defines a critical Hardy weight. 
	The weight $w$ is null-critical (and hence optimal) if additionally, 
	\begin{align*}
		\sum_{X}  G_oc  \neq m(o).
	\end{align*}
		%
	In particular, $w$ is null-critical if one of the following  conditions are satisfied:
		\begin{itemize}
			\item $c = 0$,
			\item $(b,c)$ is not stochastically complete at infinity,
			\item $(b,c)$ is stochastically complete at infinity and for some $x\in X$
			$$\lim_{t \to \infty} \sum_X m e^{-tL}1_x \neq 0.$$ 
		\end{itemize}
\end{theorem}

\begin{remark}[$c\neq 0$]
	The criterion $\sum_X G_o c \neq m(o)$ can be interpreted as a violation of the Green's formula for the potential $u=G_o$. It can be seen, \cite[Theorem~1 and Theorem~13]{HKPIV}, that the validity of such a Green's formula for potentials is tightly related to a Liouville theorem, namely the non-existence of non-trivial bounded, harmonic functions. Clearly, for $c=0$ the criterion $\sum_X G_o c \neq m(o)$ immediately gives an optimal Hardy weight. 
	
	We discuss the case of $c \neq 0$ in more detail: an easy example which shows what can go wrong is the standard Laplacian on $\N_0$ with a Dirichlet boundary condition at $0$ which gives a non-vanishing killing term at $1$, i.e., $c=1_{\{1\}}$. In this case, the Green's function is given by $G(x,1) = 1$ for all $x \in \N_0$, and we have $Gc(1)=\sum_{x \in \N_0} G_1c = c(1)=1$. 
On the other hand, $w=\mathcal{L}G^{1/2}_1/G^{1/2}_1=1_{\{1\}}$ is a Hardy weight for the graph  which equals the killing term. This is of course critical (as $\N$ without a killing term is recurrent), but not null-critical as $\sum_X G_1 w=1<\infty$. Indeed, this a special case of a more general phenomenon that whenever the graph $b$ is recurrent and $c\ne0$, then one has $Gc=1$, see \cite[Theorem~23]{HKPIV}. 

This gives a first impression of what can go wrong in the presence of a killing term. We next address the second and third bullet points in the 'in particular'-part of the theorem. 
In fact, the 
authors, together with Marcel Schmidt, have shown in \cite[Theorem~1, Theorem~8 and Theorem~26]{HKPIV} that the condition
$\sum_X G_o c = m(o)$ is equivalent to the graph $(b,c)$ being stochastically complete at infinity and at the same time, for some (equivalently all) $x \in X$ 
\[
\lim_{t \to \infty} \sum_X me^{-tL}1_x  = \lim_{t \to \infty} e^{-tL}1(x) = 0.
\]
This gives a characterization of the condition $\sum_X G_o c \neq m(o)$ in terms of stochastic completeness at infinity and the long time behavior of the heat kernel, i.e., the latter two bullet points.

For the second case, i.e.,\@ when $(b,c)$ is not stochastically complete at infinity,  this requires the graph $(b,0)$ without killing term to be stochastically incomplete,  (see \cite[Theorem~3.18]{KLW} and note that $\alpha$-subharmonic functions for $(b,c)$ are $\alpha$-subharmonic for $(b,0)$). Examples of such graphs are studied widely, see e.g. \cite{KLW13,Woj09,Woj11}. In \cite[Theorem 2]{KL12} a construction is given how a stochastically incomplete graph can become stochastically complete at infinity by adding a large killing term. This in turn informs that for stochastically incomplete graphs $(b,0)$ the graph $(b,c)$ stays stochastically incomplete at infinity if $c$ is not too large. For a characterization for spherically symmetric graphs see \cite[Theorem 9.25]{KLW}.

For the third case, i.e.,\@ when $(b,c)$ is stochastically complete at infinity, the expression  $\sum_X m e^{-tL}1_x  $ can be understood as  the total heat within the graph at time $t$ if we start with all the heat concentrated at $x$ at time $t=0$. The condition that the limit does not tend to zero	 means that not all of the heat eventually leaves the graph. The contrary again can be achieved by adding a very large killing term. 
\end{remark}

 We now turn to the fractional Laplacian for a graph, see  \cite[Section~4]{HKPII} for a detailed discussion. 
 Given a graph $(b,c)$ over a discrete measure space $(X,m)$, the {\em fractional Laplace $L^{\sigma}$} for $\sigma \in (0,1)$ is defined via the spectral calculus of the Laplacian $L$ associated to the graph. 
 Furthermore, there is a graph  $(b_{\sigma}, c_{\sigma})$ over $(X,m)$ such that $L^\sigma$ is a restriction of the formal Laplacian $\mathcal{L}^{\sigma}$ associated with $(b_{\sigma}, c_{\sigma})$, see \cite[Theorem~24]{HKPII}. It shall be noted that if $c=0$ and $b$ is stochastically complete, then $c_{\sigma}=0$ for all $\sigma \in (0,1)$, see \cite[Theorem~24]{HKPII}.
If  the graph $(b_{\sigma}, c_{\sigma})$ is transient, we denote the Green's function by $G_{\sigma}$, and write $G=G_1$.

Indeed, the Green's function $G_\sigma=G_{\sigma,o}$ is given by the {discrete Riesz kernel} at the vertex $o$ which is defined  for $\alpha >0$ as $k_\alpha = k_{\alpha,o}\colon X\to[0,\infty]$ 
\[
k_{\alpha} = \frac{1}{|\Gamma(\alpha)|} \int_0^{\infty} e^{-tL}1_o \frac{dt}{t^{1-\alpha}}.
\]
Whenever, $(b_\sigma,c_\sigma)$ is transient, we have $k_{\sigma} = G_{\sigma}$, see \cite[Theorem~25]{HKPII}.

Furthermore, we say that a graph $b$ over $(X,m)$ satisfies a {\em Nash inequality of dimension $d > 0$} if there is a constant $N > 0$ such that for all $\varphi \in C_c(X)$,
\begin{align*}
	{\|\varphi\|}_2^{2+\frac{4}{d}} \leq N \mathcal{Q}(\varphi){\|\varphi\|}_1^{\frac{4}{d}}	
\end{align*}
and a {\em Sobolev inequality of dimension $d > 2$} if there is a constant $S > 0$ such that for all $\varphi \in C_c(X)$,
\begin{align*}
	{\|\varphi\|}_{\frac{2d}{d-2}}^2 \leq S \mathcal{Q}(\varphi).
\end{align*}

\begin{theorem} \label{thm:FrLa}
		Let $b$ be a stochastically complete connected graph over $(X,m)$. If $b_\sigma$ is transient, then
		$$w_{\sigma} =  \frac{\mathcal{L}^{\sigma}k_{\sigma}^{1/2}}{k_{\sigma}^{1/2}}$$  is a null-critical (and hence optimal) Hardy weight.
		If for $\sigma \in (0,1)$, the graph $b$ satisfies a Nash inequality of dimension $d > 2\sigma$, then $b_{\sigma}$ is transient. If $b$ satisfies a Sobolev inequality of dimension $d>2$, then $b_{\sigma}$ is transient for all $\sigma \in (0,1)$.
	\end{theorem}

In the presence of known asymptotics of the discrete Riesz kernel, one also obtains asymptotic expansions for the Hardy weight $w_{\sigma}$. 
We will illustrate this for the fractional Laplacian associated with the (stochastically complete) standard lattice graph over $\Z^d$. We write $\Delta$ for the Laplacian associated with the graph with weights $b(x,y) = 1$ if $|x-y|=1$ and $b(x,y) = 0$ otherwise for $x,y\in \Z^d$.  
It has  been shown in \cite[Theorem~1.1]{HKP} that  
 for $\sigma \in (0,1]$ with $\sigma < d/2$ if $d \in \{1,2\}$, 
  the function ${w}_{\sigma, \alpha_0}: \Z^d \to (0,\infty),$
\[
    {w}_{\sigma, \alpha_0} = \frac{k_{\alpha_0- \sigma}}{k_{\alpha_0}} = \frac{\Delta^{\sigma}k_{\alpha_0}}{k_{\alpha_0}}
\]
is a null-critical Hardy weight for the fractional Laplacian $\Delta^{\sigma}$, where $\alpha_0 = (d/2 + \sigma)/2$, and one gets the expansion
\begin{align*}
{w}_{\sigma,\alpha_0} = c_{d,\sigma}|x|^{-2\sigma} + \mathcal{O}(|x|^{-2\sigma-2}) 
\end{align*}
for $|x| \to \infty$ with explicit constant 
\[
 c_{d,\sigma} = 4^{\sigma}\, \frac{\Gamma\Big( \frac{d/2+\sigma	}{2} \Big)^2}{\Gamma\Big( \frac{d/2-\sigma}{2} \Big)^2}.
\]
For the one-dimensional case $d=1$ and $\sigma \in (0,1/2)$, this result had already been proven before in \cite[Theorem~1]{KN} with slightly weaker error term  $ \mathcal{O}(|x|^{-2\sigma -1})$.

\medskip

We show next that the null-critical Hardy weight $w_{\sigma}$ from Theorem~\ref{thm:FrLa}   has the same leading term in its asymptotics as $\widetilde{w}_{\sigma,\alpha_0}$.

\begin{theorem}[Null-critical Hardy weight on $\Z^d$] \label{thm:Zd}
	Let $d \in \N$ and $\sigma \in (0,1]$ with $\sigma < d/2$ if $d \in \{1,2\}$. 
	Then the null-critical Hardy weight $w_{\sigma} = {\Delta^{\sigma} k_{\sigma}^{1/2}}/{k_{\sigma}^{1/2}}$   satisfies the asymptotics
	\[
	w_{\sigma}(x) = c_{d,\sigma}|x|^{-2\sigma} +\mathcal{O}\big( |x|^{-q} \big)
	\]
	as $|x| \to \infty$, where $q=\min\{2,d/2+3\sigma\}$ if $\sigma \in (0,1)$ and $q=3$ if $\sigma=1$. 
\end{theorem}

For $\sigma=1$ and $d \geq 3$, this result has already been obtained in \cite[Theorem~7.2]{KPP18}. Indeed, the error term in the asymptotics of $w_1$ can be improved to $q=4$, see the forthcoming paper \cite{HKPV}.
We prove the above theorem in the cases $\sigma \in (0,1)$ via estimates on the difference $w_{\sigma} - {w}_{\alpha_0,\sigma}$. For the asymptotics we use the expansion of the discrete Riesz kernel obtained in \cite[Theorem~A.1]{HKP}. We thus obtain  a short proof of the fact that the leading terms  of the two Hardy weights coincide. 

\section{The main technical lemma} \label{sec:prelim}

For a graph $(b,c)$ over $(X,m)$ with formal domain $\mathcal{F}$,  and  $\mathcal{D}$ as defined above,
one gets a bilinear form $\mathcal{Q}: \D \times \D \to \R$ acting as  
\begin{align*}
	\mathcal{Q}(f,g)=\frac{1}{2}\sum_{x,y\in X}b(x,y)(f(x)-f(y))(g(x) - g(y))+\sum_{x\in X}c(x)f(x)g(x).
\end{align*}
On the diagonal we set $\mathcal{Q}(f) = \mathcal{Q}(f,f)$ as  above. Furthermore,  $\mathcal{Q}(f,g)$ can be extended to $f,g \in \mathcal{F}$ such that not both $f$ and $g$ are in $\mathcal{D}$ but $(f(x)-f(y))(g(x)-g(y))\ge0$ and $f(x)g(x) \geq 0$ for all $x,y \in X$, so that while the series may be infinite all involved terms  are  positive. An element $u \in \mathcal{F}$ is called {\em superharmonic} if $\mathcal{L}u \geq 0$, and {\em subharmonic} if $\mathcal{L}u \leq 0$. It is called {\em harmonic} if $\mathcal{L}u=0$, i.e.,\@ $u$ is both super- and subharmonic.

We next recall two facts which were proven in \cite{HKPII} and which will be used in the proofs below. We start with following characterization of potentials in terms of the Green's operator $G$ and the formal Laplacian $\mathcal{L}$.

\begin{proposition}[Proposition~4 of \cite{HKPII}]
	\label{prop:characterisation_potentials}
	A function $u \in \F$ satisfies $\mathcal{L} 	u\in \mathcal{G}$ and $u=G\mathcal{L}u$ if and only if it is a potential, i.e., there exists $k \in \mathcal{G}$ with $u = Gk$. In that case $\mathcal{L} u = k$ and $u = G\mathcal{L} u$.
\end{proposition}

We will also subsequently make use of the following characterization of the space $\D_0$. 
	
	\begin{proposition}[Lemma~5 of \cite{HKPII}] \label{prop:charD_0}
		We have $u \in \D_0$ if and only if there is a sequence $(\varphi_n)$ in $C_c(X)$ converging pointwise to $u$ as $n \to \infty$ and satisfying
		\[
		\sup_{n\in\N} \mathcal{Q}(\varphi_n) < \infty.
		\]
	\end{proposition}
	
	The next lemma is crucial for the proof of Theorem~\ref{thm:null_critical_weights_Hardy}.
		\begin{lemma}[Main technical lemma]
			\label{lem:weak_green_formula_on_P}
			Let $\lambda:\R \rightarrow \R$ be Lipschitz continuous and such that $\lambda(0) = 0$. 
			\begin{itemize}
				\item[(a)] If $u \in \D_0$, then  $\lambda(u) \in \D_0$. 
			\end{itemize}
			Suppose now that $\lambda$ is additionally bounded and monotonically increasing. 
			\begin{itemize}
				\item[(b)] If $u \in \P\cup \mathcal{D}_0$ and $\mathcal{L} u \in \ell^1(X,m)$, then $\lambda(u) \in \D_0$ and 				\begin{align*}
					0 \leq \mathcal{Q}(u, \lambda(u)) \leq \sum_{x \in X} m\mathcal{L} u  \lambda (u).
				\end{align*} 
				In particular, $\mathcal{Q}(u,\lambda(u))$ converges absolutely.
			\end{itemize}
		\end{lemma}
		
		\begin{proof}
			To see~(a), with Lipschitz continuity of $\lambda$ and the assumption $\lambda(0)=0$, note first that we clearly have $\mathcal{Q}(\lambda \circ v) \leq L^2 \mathcal{Q}(v)$ for all $v \in \D$, where $L$ is a Lipschitz constant for $\lambda$. For $u \in \D_0$, there exists a sequence $(\varphi_n)$ in $C_c(X)$ that converges pointwise and with respect to $\mathcal{Q}$ to $u$ as $n\to \infty$, \cite[Lemma 6.5]{KLW}. 
			We observe further that  the support of $\lambda \circ \varphi_n$ is included in the support of $\varphi_n$, since $\lambda(0)=0$. 
			Thus $\big( \lambda \circ \varphi_n \big)$ is a sequence in $C_c(X)$ that converges pointwise to $\lambda \circ u$ as $n \to \infty$. 
			We obtain 
			\[
			\sup_{n \in \N} \mathcal{Q}\big( \lambda \circ \varphi_n )   \leq L^2 \mathcal{Q}(\varphi_n)  < \infty.
			\]
			We infer $\lambda \circ u \in \D_0$ from    Proposition~\ref{prop:charD_0}.
			
			We turn to the proof of~(b). We first assume $u\in \mathcal{P}$. Let $(k_n)$ be a sequence in $C_c(X)$ that converges pointwise to $\mathcal{L} u$ and such that $\abs{k_n} \leq \abs{\mathcal{L} u}$ for all $n \in \N$. Then $u_n := Gk_n \in \D_0$ by \cite[Lemma~6]{HKPII}. 
			Since, $u_n\in D_0$ with $\mathcal{L} u_n = k_n \in C_c(X)\subseteq \ell^1(X,m)$ and
			$\lambda(u_n)\in \mathcal{D}_0\cap \ell^{\infty}(X)$, as $\lambda$ is bounded, we can invoke a Green's formula, see \cite[Lemma~6.8]{KLW}, to obtain
			\begin{align*}
				\mathcal{Q}(u_n, \lambda(u_n)) = \sum_{ X} m\mathcal{L} u_n  \lambda(u_n )   =  \sum_{  X}m k_n \lambda(u_n ) \leq \norm{\lambda}_{\infty} \norm{\mathcal{L} u}_{ 1}.
			\end{align*}
			Since $u$ is a potential, by the dominated convergence theorem, the sequence $(u_n)$ converges pointwise to $u$. Since $\lambda$ is continuous, this also implies that $(\lambda(u_n))$ converges pointwise to $\lambda(u)$. Furthermore, $\lambda$ is bounded, $\mathcal{L} u \in \ell^1(X,m)$ and $\abs{\mathcal{L} u_n} = \abs{k_n} \leq \abs{\mathcal{L} u}$ and so it follows from the dominated convergence theorem that
			\begin{align*}
				\lim_{n \to \infty} \sum_{  X} m\mathcal{L} u_n  \lambda(u_n ) = \sum_{x \in X}m \mathcal{L} u \lambda(u ).
			\end{align*}
			Lastly $\lambda$ is monotonically increasing and so it holds for every $x, y \in X$ that
			\begin{align*}
				&(u(x)-u(y))(\lambda(u(x)) - \lambda(u(y))) \geq 0, \\
				&(u_n(x)-u_n(y))(\lambda(u_n(x)) - \lambda(u_n(y))) \geq 0.
			\end{align*}
			Together with $\lambda(0) = 0$, the monotonicity of $\lambda$ shows that $\lambda(t)t \geq 0$ for all $t \in \R$.
			With Fatou's lemma, we see that
			\begin{multline*}
				0\leq \mathcal{Q}(u, \lambda(u))  \leq \liminf_{n \to \infty} \mathcal{Q}(u_n, \lambda(u_n)) \\= \liminf_{n \to \infty} \sum_{  X} m\mathcal{L} u_n  \lambda(u_n )  = \sum_{ X}m \mathcal{L} u \lambda(u ),
			\end{multline*}
			as claimed. It remains to be shown that $\lambda(u) \in \D_0$. To this end, note that by Lipschitz continuity of $\lambda$, $\lambda(0)=0$, and the fact that $\lambda$ preserves signs, we have  for all $n \in \N$, 
			\begin{align*}
				\mathcal{Q}(\lambda(u_n)) \leq C \, \mathcal{Q}(u_n, \lambda(u_n)) \leq C \norm{\lambda}_{\infty} \norm{\mathcal{L} u}_{1}.
			\end{align*}
			We conclude  $\lambda(u) \in \D_0$ from Proposition~\ref{prop:charD_0}.

			For $u \in \D_0$, the claim follows directly (and is indeed an equality) as the Green's formula, \cite[Lemma~6.8]{KLW} applies directly for $u\in \mathcal{D}_0$ with $\mathcal{L} u \in \ell^1(X,m)$ and $\lambda(u) \in \D_0\cap\ell^{\infty}(X)$ which is true by (a).
		\end{proof}

	\section{Null-critical Hardy (type) weights}\label{sec:null_critical_weights}

	We extend our scope of investigation to weights that might take negative values. On the one hand, this does not complicate the considerations substantially, and on the other hand, this situation naturally arises if one uses solutions from corresponding continuum models in the supersolution construction, see e.g. \cite{hake2025optimal}.
	
	Given a graph $(b,c)$ over $(X,m)$ we say that a function $w:X \to \R$ is a {\em Hardy type weight} if a {\em Hardy type inequality} holds for $w$, i.e.\@
	\begin{align*}
		\mathcal{Q}(\varphi) \, \geq \,  \sum_{  X}m w \varphi^2  , \qquad \varphi \in C_c(X).
	\end{align*}
	A Hardy type weight $w$ is said to be {\em critical} if every $\widetilde{w} \geq w$ with $\widetilde{w} \neq w$ is not a Hardy type weight. As for Hardy weights, $w$ is a critical Hardy type weight if and only there is a {\em ground state} for $w$, i.e.,\@ a strictly positive $v \in \mathcal{F}$ such that $\mathcal{L}v = wv$, cf.\@ \cite[Theorem~5.3]{KPP20} and \cite[Theorem~14]{HKPII}. A critical Hardy type weight $w$ is {\em positive critical} if $v \in \ell^2(X,|w|m)$, i.e.,\@ the Hardy type inequality admits a minimizer. A critical Hardy type weight $w$ that is not positive critical is called {\em null-critical}. Such weights $w$ are also {\em optimal near infinity}, i.e.,\@ for every finite $K \subseteq X$ and each $\varepsilon > 0$, there is some $\varphi \in C_c(X)$ supported in $X \setminus K$ such that 
	$$\mathcal{Q}(\varphi) <  \sum_{  X}m (w+\varepsilon|w|) \varphi^2  ,$$
	cf.\@ \cite[Proposition~15]{HKPII}. 
	Clearly, every Hardy weight is a Hardy type weight.
	
	\medskip

	We now formulate the abstract main results of this paper, providing  sufficient criteria for the existence of null-critical Hardy (type) weights. 
	
	
	\begin{theorem}
		\label{thm:main_optimality_theorem}
		Let $(b,c)$ be a connected graph over  $(X,m)$. 
		Let $u \in \mathcal{F}$, $u > 0$ and let $w = {\mathcal{L} u^{\frac{1}{2}}}/{u^{\frac{1}{2}}}$. If $\mathcal{L} u \in \ell^1(X,m)$ and if $u \in \D_0\cup \P$, then $w$ is a critical Hardy type weight. If additionally $$\sum_{ X} m\mathcal{L} u  \neq \sum_{  X} cu,$$ then $w$ is  null-critical.
	\end{theorem}
	
	Before we give the proof to this theorem, we show how Theorem~\ref{thm:null_critical_weights_Hardy} and Theorem~\ref{thm:optGF} can be derived	given this theorem.

	\begin{proof}[Proof of Theorem~\ref{thm:null_critical_weights_Hardy}] By \cite[Theorem~1]{HKPIV}, the condition for null-criticality $G(c/m)(x) < 1$ is equivalent to $\sum_X m\mathcal{L}u \neq  \sum_X cu$ for every superharmonic  $u\in\P$. Thus,  the result  follows immediately from Theorem~\ref{thm:main_optimality_theorem}, together with the observation that $u^{1/2}$ is superharmonic if $u$ is, cf.\@ \cite[Corollary~2.3]{KPP18}.
	\end{proof}

	We can explicitly describe any potential in terms of the Green kernel. This yields the following.
	
	\begin{corollary} \label{cor:GF}
		\label{thm:optimal_weights_from_green_function}
	If $k \in \ell^1(X,m)$, then $k\in \mathcal{G}$. If furthermore $0\neq k \geq 0$, then $w =  {\mathcal{L} \left( Gk \right)^\frac{1}{2}}/{\left( Gk \right)^\frac{1}{2}}$ is a  critical Hardy weight. Moreover, $w$ is null-critical if
	\[
	\sum_X m k \neq \sum_{  X} c(Gk).
	\]
	In particular, $w$ is null-critical if $c=0$.
	\end{corollary}

	\begin{proof}
			If $k \in \ell^1(X,m)$, then by $G(x,y)\leq m(y)/m(x)G(x,x)$ \cite[Theorem B.1]{KLSS} and \cite[Theorem~6.26~(a)]{KLW}, we obtain 
			$k\in \mathcal{G}$.			
		If $k \in \mathcal{G}$ then $u = Gk \in \P$ and $\mathcal{L} u = k$ by Proposition~\ref{prop:characterisation_potentials}. So, connectedness,  $\mathcal{L} u=k \in \ell^1(X,m)$ and $0\neq k \geq 0$  imply that $u > 0$ and superharmonicity of $u$, \cite[Theorem~6.26x]{KLW}. Moreover,  it is easy to see that the same holds for $u^{1/2}$, cf.~\cite[Corollary~2.3]{KPP18}. The  statements now follow from Theorem~\ref{thm:main_optimality_theorem}, while noting that 
		\[
		0 \neq \sum_{  X}m  k = \sum_{  X}m \mathcal{L}Gk.  \hfill\qedhere
		\]
	\end{proof}
	
We can now prove Theorem~\ref{thm:optGF}. 

\begin{proof}[Proof of Theorem~\ref{thm:optGF}] 
By \cite[Theorem~1 or Theorem~26]{HKPIV},  $G(c/m)=1$   is equivalent to stochastic completeness at infinity and $\sum_X m e^{-tL}1_x\to0$ as $t\to\infty$.  Hence, in all cases stated in the theorem, we have $G(c/m) < 1$, cf. \cite[Lemma~4]{HKPIV} for the dichotomy of $G(c/m)=1$ and $G(c/m)<1$. Now, the  result follows immediately from Theorem~\ref{thm:null_critical_weights_Hardy}.
\end{proof}
	
\begin{remark}[Optimality of the criteria] The theorem above does not characterize null-criticality of the Hardy weights. In  the case $c=0$ and $m=1$, the dissertation thesis \cite[Proposition~2.29]{hake2025optimal} of the first named author  gives some additional insight: Given a  superharmonic $u \in \mathcal{F}$, $u > 0$   such that $w = {\mathcal{L} u^{\frac{1}{2}}}/{u^{\frac{1}{2}}}$ is a critical Hardy weight, one has the following:
	\begin{enumerate}
		\item[(a)] If $u \in \D$ then $u \in \D_0$. 
		\item[(b)] If $u$ is bounded then $u$ is a potential.
	\end{enumerate}
\end{remark}

	The remainder of this section is devoted to the proof of the criteria on critical and null-critical Hardy (type) weights, especially Theorem~\ref{thm:main_optimality_theorem}.  
	We introduce the following families of functions. For all $T > 1$ we define the functions $\lambda_T^+, \lambda_T^-:\R \rightarrow \R$ by
	
	\begin{align*}
		\lambda_T^+(t) = \int_{\frac{1}{T}}^{1} \frac{1_{[0,t]}(s)}{s} ds, \qquad \lambda_T^-(t) = \int_1^{T} \frac{1_{[0,t]}(s)}{s} ds.
	\end{align*}
	
	\begin{lemma}
		\label{lem:properties_of_lambda_T}
		For every $T > 1$ the functions $\lambda_T^\pm$ are monotonically increasing and Lipschitz continuous. They satisfy $0 \leq \lambda_T^\pm \leq \log T$ and $\lambda_T^\pm(0) = 0$. Furthermore, for all $r, t > 0$,
		\begin{align*}
			t^\frac{1}{2} r^\frac{1}{2} \abs{\lambda_T^\pm(t) - \lambda_T^\pm(r)} \leq \abs{t-r}. 
		\end{align*}
	\end{lemma}
	
	\begin{proof}
		We only prove these statements for $\lambda_T^+$. The proof for $\lambda_T^-$ is similar.
			Let $r, t \in \R$, $t > r$ be arbitrary. Then
		\begin{align*}
			\lambda_T^+(t) - \lambda_T^+(r) &= \int_{\frac{1}{T}}^{1} \frac{1_{(r,t]}(s)}{s}  \, ds \in \left[ 0, T(t-r) \right],
		\end{align*}
		which implies that $\lambda_T^+$ is monotonically increasing and Lipschitz continuous with Lipschitz constant $T$. Clearly, $\lambda_T^+(t) = 0$ if $t \leq 0$ and $\lambda_T^+(t) = \log T$ if $t \geq 1$. Thus, $0 \leq \lambda_T^+ \leq \log T$ and $\lambda_T^+(0) = 0$. Next, assume additionally that $t> r > 0$. Then, we can see similarly as before that
		\begin{align*}
			\lambda_T^+(t) - \lambda_T^+(r) &= \int_{\frac{1}{T}}^{1} \frac{1_{(r,t]}(s)}{s} \, ds \leq \int_r^t \frac{1}{s} \, ds = \log \left( \frac{t}{r} \right).
		\end{align*}
		
	Consider the function $\lambda:(0,\infty) \rightarrow \R$, $\lambda(\xi) = \xi - \xi^{-1} - 2 \log(\xi)$. It satisfies 
		\begin{align*}
			\lambda'(\xi) = 1 + \xi^{-2} - 2\xi^{-1} = (1-\xi^{-1})^2 \geq 0
		\end{align*} for all $\xi \in (0, \infty)$
		and $\lambda(1) = 0$. Therefore, $\lambda(\xi) \geq 0$ if $\xi > 1$. If we choose $\xi = \frac{t}{r} > 1$, then we can calculate 
		\begin{align*}
			0 \leq \lambda \! \left( \xi^\frac{1}{2} \right)  
			= \frac{t^\frac{1}{2}}{r^\frac{1}{2}} - \frac{r^\frac{1}{2}}{t^\frac{1}{2}} - \log \left( \frac{t}{r} \right)
			\leq \frac{t^\frac{1}{2}}{r^\frac{1}{2}} - \frac{r^\frac{1}{2}}{t^\frac{1}{2}} - \left( \lambda_T^+(t) - \lambda_T^+(r) \right). 
		\end{align*}
		Hence,
		\begin{align*}
			t^\frac{1}{2} r^\frac{1}{2} ( \lambda_T^+(t) - \lambda_T^+(r) ) \leq t-r 
		\end{align*}
		and, as $\lambda_T^+$ is monotone increasing, we have, for all $r, t > 0$,
		\begin{equation*}
			t^\frac{1}{2} r^\frac{1}{2} \abs{\lambda_T^+(t) - \lambda_T^+(r)} \leq \abs{t-r}.\hfill\qedhere
		\end{equation*}
	\end{proof}
	
	For a strictly positive $v \in \mathcal{F}$ and $f \in C(X)$,
	 we define 
	\[
	\mathcal{Q}_v(f) = \frac{1}{2} \sum_{x,y \in X} b(x,y)v(x)v(y) \left( {f(x)} - {f(y)}  \right)^2,
	\] 
	which may be infinite. 
	
	\begin{lemma}
		\label{lem:F_T^pm_u_in_D_0}
		Let $u \in \D_0\cup  \P$, $u > 0$ and $\mathcal{L} u \in \ell^1(X,m)$. Let $v = u^{\frac{1}{2}}$ and let $T > 1$. Then $\lambda_T^\pm (u) \in \D_0$ and 
		\begin{align*}
			\mathcal{Q}_v(\lambda_T^\pm(u))  \leq \mathcal{Q}(u, \lambda_T^\pm(u)) \leq \log T \, \norm{\mathcal{L} u}_{1}.
		\end{align*} 
	\end{lemma}
	
	\begin{proof} The fact $\lambda_T^\pm (u) \in \D_0$ follows from Lemma~\ref{lem:weak_green_formula_on_P}. Furthermore, it follows from the previous lemma that for all $x, y \in X$ 
		\begin{align*}
			v(x)v(y) (\lambda_T^\pm(u(x))-\lambda_T^\pm(u(y)))^2 &= u^\frac{1}{2}(x)u^\frac{1}{2}(y) \abs{\lambda_T^\pm(u(x))-\lambda_T^\pm(u(y))}^2 \\
			&\leq \abs{u(x)-u(y)} \abs{\lambda_T^\pm(u(x))-\lambda_T^\pm(u(y))} \\
			&= (u(x)-u(y)) (\lambda_T^\pm(u(x))-\lambda_T^\pm(u(y))).
		\end{align*}
		The last equality holds because $\lambda_T^\pm$ is monotonically increasing. Thus,
		\begin{align*}
			\mathcal{Q}_v(\lambda_T^\pm(u))
			&\leq  \frac{1}{2} \sum_{x, y \in X} b(x,y) (u(x)-u(y)) (\lambda_T^\pm(u(x))-\lambda_T^\pm(u(y)))\\
			&= \mathcal{Q}\big( u,\, \lambda_T^{\pm}(u) \big)\le \sum_{  X}  m{ \mathcal{L} u  \lambda_T^\pm(u ) }  \leq \log T \, \norm{\mathcal{L} u}_{1} ,
		\end{align*}
		by Lemma~\ref{lem:weak_green_formula_on_P} and Lemma~\ref{lem:properties_of_lambda_T}.
		This finishes the proof.
	\end{proof}
	
	Next, we  define $$\lambda_T := \lambda_T^+ - \lambda_T^-$$ for  $T > 1$. Then $$0 \leq \lambda_T \leq \log T$$ since $\lambda_T^+$ and $\lambda_T^-$ are monotonically increasing with $0 \leq \lambda_T^\pm \leq \log T$ and $\lambda_T^+(1) = \log T$, $\lambda_T^-(1) = 0$.

	\begin{lemma}
		\label{lem:technical_inequality}
		Let $u \in \F$, $u > 0$, $v = u^\frac{1}{2}$ and $T > 1$. If $\varphi \in C_c(X)$ is such that $0 \leq \varphi \leq \lambda_T(u)$, then
		\begin{align*}
			\mathcal{Q}_v(\varphi)  \leq \sqrt{2}T \, \mathcal{Q}(\varphi) + \frac{2(\log T)^2}{\log 2}   \mathcal{Q}(u, \lambda_{2T}^-(u))   .
		\end{align*}
	\end{lemma}
	
	\begin{proof}
		We claim that  for all $x, y \in X$ that either the first inequality
		\begin{align*}
			v(x)v(y) \leq \sqrt{2}T 
		\end{align*}
		or the second inequality
		\begin{align*}
		 v(x)v(y) ( \varphi(x) - \varphi(y) )^2 \leq 	\frac{\sqrt{2}(\log T)^2}{\log 2} (u(x)-u(y))( \lambda_{2T}^-(u(x))-\lambda_{2T}^-(u(y)))
		\end{align*}
		holds.
	Given this claim then the statement of the lemma follows because  
		\begin{multline*}
			\mathcal{Q}_v(\varphi)  
			\leq \frac{1}{2} \sum_{x, y \in X} b(x,y) \Big[ \sqrt{2}T ( \varphi(x) - \varphi(y) )^2  \\ + \frac{\sqrt{2}(\log T)^2}{\log 2} (u(x)-u(y))(\lambda_{2T}^-(u(x))-\lambda_{2T}^-(u(y))) \Big] \\
			 \leq \sqrt{2}T \, \mathcal{Q}(\varphi) + \frac{\sqrt{2}(\log T)^2}{\log 2} \mathcal{Q}(u, \lambda_{2T}^-(u)).
		\end{multline*}
		Note that we used non-negativity of $u$ and $\lambda^{-}_{2T}$ in the last step.

		To prove the claim, fix $x, y \in X$. Since both inequalities in the claim are symmetric in $x, y$ we can assume without loss of generality that $u(x) \geq u(y)$. We consider three cases.
		
		First, if $u(y) > T$, then  $\lambda_T^+(u(y)) = \log T=\lambda_T^-(u(y)) $ and so $\lambda_T(u(y)) = 0$. This implies $\varphi(y) = 0$ because $0 \leq \varphi \leq \lambda_T(u)$. Since $u(x) \geq u(y)$, it similarly follows that $\varphi(x) = 0$. Therefore, 
		$			( \varphi(x) - \varphi(y) )^2 = 0$ .
		On the other hand, 
		\begin{align*}
			(u(x)-u(y))(\lambda_{2T}^-(u(x))-\lambda_{2T}^-(u(y))) \geq 0
		\end{align*}
		is always true because $\lambda_{2T}^-$ is  increasing. This gives the second inequality.
		
		Secondly, we consider  $u(x) \geq 2T\ge T\ge u(y)$. Then it follows from the definition and the monotonicity of $\lambda_{2T}^-$ that
		\begin{align*}
			 \lambda_{2T}^-(u(y)) \leq \lambda_{2T}^-(T) = \log T, \qquad
			 \lambda_{2T}^-(u(x)) \geq \lambda_{2T}^-(2T) = \log (2T)
		\end{align*}
		and so $$\lambda_{2T}^-(u(x)) - \lambda_{2T}^-(u(y)) \geq \log 2.$$ 
		We  also see that
		\begin{align*}
			u(x) - u(y) \geq \frac{1}{2} u(x) \geq \frac{1}{\sqrt{2}} v(x) v(y).
		\end{align*}
	 Since $0 \leq \varphi \leq \lambda_T(u) \leq \log T$ by Lemma~\ref{lem:properties_of_lambda_T}, it always holds that
		\begin{align*}
			(\varphi(x) - \varphi(y))^2 \leq (\log T)^2.
		\end{align*}
		Putting those inequalities together, we get the second inequality		\begin{align*}
			v(x)v(y) ( \varphi(x) - \varphi(y) )^2 \leq \frac{\sqrt{2}(\log T)^2} {\log 2}(u(x)-u(y))(\lambda_{2T}^-(u(x))-\lambda_{2T}^-(u(y)))
		\end{align*}
		whenever $u(x) \geq 2T\ge T\ge u(y)$. 
		
	The remaining case is $u(y) \leq T$, $u(x) < 2T$ for which  we can estimate
		\begin{equation*}
			v(x)v(y) \leq \sqrt{T} \sqrt{2T} = \sqrt{2}T.\hfill\qedhere
		\end{equation*}
	\end{proof}

	\begin{lemma}
		\label{lem:F_T_u_in_D_0(h_v)}
		Let $u \in \D_0\cup \P$, $u > 0$ and $\mathcal{L} u \in \ell^1(X,m)$, let $v = u^\frac{1}{2}$ and $T > 1$. Then $\lambda_T(u) \in \D_0(\mathcal{Q}_v)$ and $$\mathcal{Q}_v(\lambda_T(u))  \leq 4 \log T \, \norm{\mathcal{L} u}_{1}.$$
	\end{lemma}
	
	\begin{proof}
		By Lemma~\ref{lem:F_T^pm_u_in_D_0} the functions $\lambda_T^+(u)$ and $\lambda_T^-(u)$ are in $\D_0$ and, therefore, $\lambda_T(u) = \lambda_T^+(u) - \lambda_T^-(u) \in \D_0$. Let $(\varphi_n)$ be a sequence in $C_c(X)$ that converges pointwise and with respect to $\mathcal{Q}$ to $\lambda_T(u)$ and such that $0 \leq \varphi_n \leq \lambda_T(u)$ for all $n \in \N$ which exists by \cite[Lemma 6.6]{KLW}. 
		
		By Lemma~\ref{lem:technical_inequality} we have for all $n \in \N$ that
		\begin{align*}
			\mathcal{Q}_v(\varphi_n)  \leq \sqrt{2}T \, \mathcal{Q}(\varphi_n) + \frac{2(\log T)^2}{\log 2} \mathcal{Q}(u, \lambda_{2T}^-(u)).
		\end{align*}
We have $\mathcal{Q}(\varphi_n) \to \mathcal{Q}(\lambda_T(u))$ and so $\sup_n \mathcal{Q}_v(\varphi_n) < \infty$ which implies  $\lambda_T(u) \in \D_0(\mathcal{Q}_v)$  by Propostion~\ref{prop:charD_0}. 

Finally, it follows with Lemma~\ref{lem:F_T^pm_u_in_D_0} that
		\begin{equation*}
			Q_v(\lambda_T(u)) 
			\leq 2 \mathcal{Q}_v\big( \lambda_T^+(u) \big) \,+\, 2 \mathcal{Q}_v\big( \lambda_T^-(u) \big) 
			\leq  4 \log T \norm{\mathcal{L} u}_{1}.\hfill\qedhere
		\end{equation*}
	\end{proof}

	The lemmas above  yield the main technical tool for the proof of Theorem~\ref{thm:main_optimality_theorem}.

	\begin{proposition}
		\label{prop:crit} Let $u \in \F$, $v = u^\frac{1}{2}>0$   and  $\lambda_T(u) \in \D_0(\mathcal{Q}_v)$ for all $T > 1$. If there is a constant  $C >0$ such that 
		\begin{align*}
			\mathcal{Q}_{v}\big( \lambda_T(u) \big) \leq C \log T
 		\end{align*}
 	for all $T > 1$, then  $w = {\mathcal{L} u^{\frac{1}{2}}}/{u^{\frac{1}{2}}}$ is a critical Hardy-type weight.
	\end{proposition}
	\begin{proof}
		Since $\lambda_T(u) \in \D_0(\mathcal{Q}_v)$,  there exists a sequence $(\varphi_n)$ in $C_c(X)$ that converges to $\lambda_T(u)$ pointwise and with respect to $\mathcal{Q}_v$. We can calculate 
		\begin{align*}
			m(\tilde{w} - w) (v\varphi_n)^2 (x)\leq \sum_X m(\tilde{w} - w)(v\varphi_n)^2  = (\mathcal{Q}-w)(v\varphi_n) - (\mathcal{Q}-\tilde{w})(v\varphi_n)
		\end{align*}
		for any $n \in \N$.
		If $\tilde{w}\ge w$ is a Hardy-type weight it follows, then  $(\mathcal{Q}-\tilde{w})(v\varphi_n) \geq 0$ and by the ground state transform, see e.g.\@ \cite[Proposition~4.8]{KPP20}, it holds that $(\mathcal{Q}-w)(v\varphi_n) = \mathcal{Q}_v(\varphi_n)$. Therefore,
		\begin{align*}
		m(\tilde{w} - w) (v\varphi_n)^2 (x) \leq \mathcal{Q}_v(\varphi_n).
		\end{align*}
		As $n$ tends to infinity the left hand side converges to $m(\tilde{w} - w) (v\lambda_T(u))^2 (x)$ and the right hand side to $\mathcal{Q}_v(\lambda_T(u))$. 	
		Thus, by assumption we conclude, for all $T > 1$,
		\begin{align*}
			m(\tilde{w} - w) (v\lambda_T(u))^2 (x) \leq \mathcal{Q}_v(\lambda_T(u)) \leq C\log T  .
		\end{align*}
				Since $T > 1$ was arbitrary, we may assume that $T^{-1} < u(x) < T$. If $u(x) \leq 1$ then $\lambda_T^+(u(x)) = \log(u(x)) + \log T$ and $\lambda_T^-(u(x)) = 0$. And if $u(x) > 1$ then $\lambda_T^+(u(x)) = \log T$ and $\lambda_T^-(u(x)) = \log(u(x))$. In either case $$\lambda_T(u(x)) = \log T - \abs{\log(u(x))}.$$ 
		It follows as $v^2 = u$ that
		\begin{align*}
			0\leq (\tilde{w} - w)(x) \leq \frac{C}{u(x)m(x)} \frac{\log T}{(\log T - \abs{\log(u(x))})^2}.
		\end{align*}
		Now, the right hand side tends to $0$ as $T \to \infty$ and so we must have $(\tilde{w} - w)(x) = 0$. And since $x$ was arbitrary it follows that $\tilde{w} = w$ and $w$ is critical.
	\end{proof}
	
	We are now in position to prove the main result of this section.
	
	\begin{proof}[Proof of Theorem~\ref{thm:main_optimality_theorem}]
		Let $u \in \D_0\cup \P$, $u > 0$  such that $\mathcal{L} u \in \ell^1(X,m)$.  By Lemma~\ref{lem:F_T_u_in_D_0(h_v)} we have $\lambda_T(u) \in \D_0(\mathcal{Q}_v)$, as well as
		\[
		\mathcal{Q}_v(\lambda_T(u)) \leq 4\log T \|\mathcal{L}u\|_1
		\] 
		for all $T > 1$.
		 The fact that $w = {\mathcal{L} u^{\frac{1}{2}}}/{u^{\frac{1}{2}}}$ is critical follows from Proposition~\ref{prop:crit} above. Moreover, \cite[Proposition~21 and Proposition~22]{HKPII}  states
		that $u^{1/2} \in \mathcal{D}_0$ and $\sum_X \mathcal{L}u \,m = \sum_{X} cu $ are necessary conditions for $w$ to be positive critical. Hence, the inequality	 $\sum_{X} \mathcal{L} u\,m \neq \sum_{X} cu$ implies null-criticality of $w$. This finishes the proof.
	\end{proof}

	\section{Fractional Laplacian} \label{sec:fractional}

	In this section we apply the optimality criteria from the previous section to the fractional Laplacian on graphs. For the self-adjoint operator $L$ associated with a graph $(b,c)$ over $(X,m)$, we   define $L^\sigma$,  $\sigma \in (0,1)$,  via the spectral theorem, and observe that it can be represented as 
	\[
	L^{\sigma}f = \frac{1}{|\Gamma(-\sigma)|} \int_0^{\infty} \big( I - e^{-tL} \big)\,f(x)\frac{dt}{t^{1+\sigma}} , 
	\]
	for $ f \in D(L^{\sigma}),$ and $x \in X$. It was shown in \cite{HKPII} that $L^{\sigma}$ is the restriction of some formal Laplacian $\mathcal{L}^{\sigma}$ on a formal domain $ \mathcal{F}^{\sigma} $ associated to a graph $(b_{\sigma},c_{\sigma})$  over $(X,m)$.  The edge weights $b_{\sigma}$ are given for $x\neq y$ by
\begin{align*} 
		b_{\sigma}(x,y)&=\frac{1}{|\Gamma(-\sigma)|}\int_{0}^{\infty}m(x)e^{-tL}1_y(x)\frac{d t}{t^{1+\sigma}}\\
		c_\sigma(x)&=\frac{1}{|\Gamma(-\sigma)|}\int_0^\infty m(x)\bigl(1- e^{-tL}1(x)\bigr)\,	\frac{dt}{t^{1+\sigma}}.
	\end{align*}
Clearly, it turns out that $c_\sigma=0$ whenever the original graph is stochastically complete. Moreover, as the semigroup $e^{-tL}$ is positivity improving due to connectedness, we obtain strict positivity of $ b_{\sigma}(x,y) $ for all $ x,y\in X $, $x\neq y$.
	
Theorem~\ref{thm:optGF} can now be applied directly  to the fractional Laplacian. We always obtain a critical Hardy weight via the Green's function whenever $(b_{\sigma},c_{\sigma})$ is transient and null-criticality is subject to the  criteria involving stochastic completeness (at infinity).

We now specialize ourselves to the case of stochastic completeness. To show  Theorem~\ref{thm:FrLa}, we need one further ingredient, which determines transience of the fractional graph in terms of a Nash inequality.

\begin{theorem}[Nash/Sobolev inequality and transience] \label{thm:Nash}
	Let $b$ be a graph over $(X,m)$. 
	\begin{enumerate}[(a)]
		\item If $b$ satisfies a Nash inequality of dimension $d > 0$, then $b_{\sigma}$ is transient if $\sigma < \min\{ d/2,1\}$.
		\item If $b$ satisfies a Sobolev inequality of dimension $d > 2$, then $b_{\sigma}$ is transient for all $\sigma \in (0,1]$.
	\end{enumerate}
	
\end{theorem}
\begin{proof}
	In order to show transience, we show that the Green kernel $G_{\sigma}$ is finite at some point $o\in X$. In \cite[Theorem~25]{HKPII} it is shown that 
	$$G_{\sigma}(o,o) = \frac{1}{\Gamma(\sigma)} \int_0^{\infty} e^{-tL}1_o(o) \frac{dt}{t^{1-\sigma}}.$$
	The integral always converges at $0$ for $\sigma \in (0,1)$, and so we only need to ensure convergence at $\infty$. Now, the Nash inequality of dimension $d$ implies that 
\begin{align*}
	\sup_{x \in X}\frac{1}{m(x)} e^{-tL}1_x(x) &\leq C t^{-\frac{d}{2}}.
\end{align*}
by \cite[Theorem~2.1]{CKS} and the Sobolev inequality implies the same estimate for $d>2$ by \cite[Theorem~2.4.2]{dav89}. Hence, the integral converges at $\infty$ if $\sigma < d/2$. This finishes the proof.
\end{proof}

We have now everything in place to prove Theorem~\ref{thm:FrLa} from the introduction.
	
	\begin{proof}[Proof of Theorem~\ref{thm:FrLa}]
	As already discussed above, stochastic completeness of $b$ implies $c_{\sigma} = 0$. Furthermore, by the theorem above, $b_{\sigma}$ is transient if $\sigma < d/2$. Since $k_{\sigma,o} = G_{\sigma}(\cdot,o)$ by \cite[Theorem~25]{HKPII}, the 		 assertion now follows from Theorem~\ref{thm:optGF}.
	\end{proof}

We now consider the standard lattice graph on $\Z^d$ with fractional Laplacian $\Delta^{\sigma}$, $\sigma \in (0,1]$, as defined in the introduction. We denote the graph associated with $\Delta^{\sigma}$ by $b_{\sigma}$. In this situation, the Riesz kernel $k_{\alpha}$ can be extended to all $\alpha \in (-1,d/2)$, and $b_\sigma(x,y) = k_{-\sigma}(y-x)$, for $x\neq y$ in $\Z^d$, cf.\@ \cite{HKP}. 
It is also well known that the graph $b_{\sigma}$ is transient for $\sigma < d/2$.	 Thus, Theorem~\ref{thm:FrLa} is applicable in this situation. Moreover, one has the  following asymptotics of the discrete Riesz kernel. 

\begin{proposition}[{\cite[Theorem~A.1]{HKP}}] \label{thm:RKasymp}
	On $\Z^d$, for $\alpha \in (-1,d/2)$, we have 
	\begin{align*}
		k_{\alpha}(x) = C_{d,\alpha}|x|^{-d + 2\alpha} + \mathcal{O}\big( |x|^{-d+2\alpha-2} \big),
	\end{align*}
	as $|x| \to \infty$ with constant 
	\[
	C_{d,\alpha} = \frac{4^{-\alpha}\Gamma(d/2 - \alpha)}{\pi^{d/2} |\Gamma(\alpha)|}.
	\]
\end{proposition}

In the following, we set $$\alpha_0 = \frac{d/2 + \sigma}{2}.$$ 

\begin{lemma} \label{lem:sqrGF}
	Let $\sigma \in (0,1)$. 
	As $|x| \to \infty$, one has 
	\[
	k_{\sigma}^{1/2}(x) = \sqrt{C_{d,\sigma}}|x|^{-d + 2\alpha_0} + \mathcal{O}\Big( |x|^{-d + 2\alpha_0 - 2} \Big).
	\]
\end{lemma}

\begin{proof}
	By Proposition~\ref{thm:RKasymp}, we have that $k_{\sigma}(x) = C_{d,\sigma}|x|^{-d+2\sigma} + \mathcal{O}(|x|^{-d+2\sigma-2})$ as $|x| \to \infty$. Taking the Taylor expansion of $(1+t)^{1/2} = 1 + t/2 - t^2/8 + \ldots$, we find 
	\begin{align*}
	k^{1/2}_{\sigma}(x) &= \sqrt{C_{d,\sigma}}|x|^{(-d+2\sigma)/2}\big( 1 + \mathcal{O}(|x|^{-2}) \big)^{1/2} \\ 
	&= \sqrt{C_{d,\sigma}}|x|^{(-d+2\sigma)/2}\big( 1 + \mathcal{O}(|x|^{-2}) \big)
	\end{align*}
	for $|x| \to \infty$. Since $-d+2\alpha_0 = (-d + 2\sigma)/2$, this shows the claim. 
\end{proof}

\begin{lemma}
	Let $\sigma \in (0,1)$.
	Set $f = g - h$, where $g=k_{\sigma}^{1/2}/\sqrt{C_{d,\sigma}}$ and $h=k_{\alpha_0}/C_{d,\alpha_0}$. We have 
	\[
	\sum_{y \in \Z^d} b_{\sigma}(x,y) \big( | f(x) | + | f(y) | \big)= \mathcal{O}(|x|^{-d+2\alpha_0 -q}), 
	\]
	as $|x| \to \infty$, 
	where $q=\min\{2,d/2+3\sigma\}$. 
	\end{lemma}
	
	\begin{proof}
		We abbreviate $\beta = d-2\alpha_0$. It follows from Proposition~\ref{thm:RKasymp} and Lemma~\ref{lem:sqrGF} that $f(x) = \mathcal{O}(|x|^{-\beta-2})$. In particular, there exists some $L >0$ such that $f(x) \leq L|x|^{-\beta-2}$ for all $x \neq 0$. Since the sum $\sum_y b_{\sigma}(x,y)=\sum_y k_{-\sigma}(y-x)$ is independent of $x$ on $\Z^d$, it is clear that $\sum_y b_{\sigma}(x, y) |f(x)| = \mathcal{O}(|x|^{-\beta-2})$. 
		Next, we are now going to estimate the  sum $\sum_{y} k_{-\sigma}(y-x)|f(y)|$. To this end, for $x \in \Z^d$, we will sum up separately over the two regions
		\[
		Y_x^1 = \big\{ y \in \Z^d\mid  |y| \geq (1/2)|x| \big\},\quad Y_x^2 = \big\{ y \in \Z^d\mid |y| \leq (1/2)|x| \big\}.
		\]
		For $y \in Y_x^1$, by summability of the weights $k_{-\sigma}$ and with  $|f(y)| \leq L|y|^{-\beta-2}$ for $y \neq 0$, we get for some $C > 0$ that
		\begin{align*}
			\sum_{y \in Y_x^1} k_{-\sigma}(y-x) |f(y)| &\leq C |x|^{-\beta-2} \sum_{y \in \Z^d} k_{-\sigma}(y-x) 
		= \mathcal{O}\big( |x|^{-\beta-2} \big)
		\end{align*}
		as $|x| \to \infty$. For $y \in Y_x^2$, we get with the asymptotic expansion on $k_{-\sigma}$, $|f(y)| \leq L |y|^{-\beta-2}$ for $y \neq 0$,
		 and estimating the sum by an integral 
		\begin{align*}
			\sum_{y \in Y_x^2}k_{-\sigma}(y-x)|f(y)|
			&= \, 	\sum_{y \in Y_x^2 \setminus \{0\}}k_{-\sigma}(y-x)|f(y)|\, + \, |f(0)|k_{-\sigma}(x) \\
			&\leq
			 \, |x|^{-d-2\sigma}\, \int_1^{|x|} r^{d-1} r^{-\beta-2}\,dr \, + \, \mathcal{O}\big( |x|^{-(d+2\sigma)} \big),
		\end{align*}
		as $|x| \to \infty$. Note that  $d-\beta-3 = -1$ is only possible if $d=3$ and $\sigma = 1/2$, and then the integral grows as $\log |x|$. We have $ \beta + q = 3 < 3+\sigma =  d+\sigma $ in this case. Otherwise, the integral is $\mathcal{O}(|x|^{d-2-\beta})$ or $\mathcal{O}(1)$.
		 In all cases,  the term in the above line is $\mathcal{O}(|x|^{-\beta-q})$ 
		 since 
		 $q \leq 2$ and $\beta + q \leq d + 2
		 \sigma$.  This finishes the proof.  
		\end{proof}
		
		We have everything together to prove Theorem~\ref{thm:Zd}.
		
		\begin{proof}[Proof of Theorem~\ref{thm:Zd}]
			For the case $d \geq 3$ and $\sigma=1$, we refer to \cite[Theorem~7.2]{KPP18}. So let $\sigma \in (0,1 \wedge d/2)$. 
			Let $g$, $h$ and $f$ be the functions of the previous lemma. As remarked above, we can apply Theorem~\ref{thm:FrLa}, and, thus, we obtain that $w_{\sigma} = \Delta^{\sigma} g / g$ is a null-critical Hardy weight. Furthermore, by \cite[Theorem~1.1]{HKP}, ${w}_{\alpha_0, \sigma} = \Delta^{\sigma}h/h = c_{d,\sigma}|x|^{-2\sigma} + \mathcal{O}(|x|^{-2-2\sigma})$. Now the previous lemma gives $|\Delta^{\sigma} f| = \mathcal{O}(|x|^{-d+2\alpha_0-q})$. Since $g$ and $h$ both have the asymptotic behavior $ |x|^{-d + 2\alpha_0} + \mathcal{O}(|x|^{-d+2\alpha_0-2})$ by Lemma~\ref{lem:sqrGF} and Proposition~\ref{thm:RKasymp}, we arrive at 
			\begin{align*}
				\big| w_{\sigma}(x) -  {w}_{\alpha_0, \sigma}(x) \big| &= \left| \frac{\Delta^{\sigma} f(x)}{g(x)} + \frac{\Delta^{\sigma} h(x)}{h(x)}\left(\frac{h(x)-g(x)}{g(x)}\right) \right| \\
				&\leq  \mathcal{O}(|x|^{-q}) + \mathcal{O}(|x|^{-2-2\sigma})
			\end{align*}
			as $|x| \to \infty$. This finishes the proof. 
		\end{proof}

\section{Extension to Schrödinger operators}\label{sec:Schr}
In this section we discuss how to extend the abstract main result, Theorem~\ref{thm:null_critical_weights_Hardy} to Schrödinger operators. Let $b$ be a connected graph over $(X,m)$ and let $q$ be a potential, i.e.,\@ a function $q : X \to \R$. We assume that the Schrödinger form $\mathcal{Q}+q$ defined on $C_c(X)$ satisfies
\begin{align*}
	(\mathcal{Q}+q)(\varphi) = \mathcal{Q}(\varphi) + \sum_X m q \varphi ^2 \ge0.
\end{align*}
This form is intimately related to the Schrödinger operator $$\mathcal{H}=\mathcal{L}+q$$ defined on the formal domain $\mathcal{F}$. A Hardy weight for $\mathcal{Q}+q$ is a function $w : X \to [0,\infty)$ such that $(\mathcal{Q}+q)(\varphi)\ge\sum_X mw\varphi^2$ for $\varphi\in C_c(X)$.  We say that $\mathcal{H}$ is \emph{subcritical} if it allows for a non-trivial Hardy weight. In this case there exists a strictly positive Green's function with an associated Green's operator $G$ as defined above. For a subcritical operator $\mathcal{H}$, we define the spaces $\mathcal{D}_0$ and $\mathcal{P}$ as above, cf.~\cite[Section~5.2]{KPP18}.

\medskip

We will show the following extension of Theorem~\ref{thm:main_optimality_theorem}.

\begin{theorem}\label{thm:main_optimality_theorem_Schr} Assume $\mathcal{H}$ is subcritical.
	Let $u \in \D_0\cup \P$, $v\in \mathcal{F}$, $u,v > 0$  such that $\mathcal{H}v=0$ and $0\lneq \mathcal{H}u \in \ell^1(X,vm)$.
	Then the function $w = {\mathcal{H} (uv)^{\frac{1}{2}}}/{(uv)^{\frac{1}{2}}}$ is a null-critical Hardy  weight for $\mathcal{Q}+q$. 
\end{theorem}

To prove the theorem, we employ the ground state transform with respect to $v$ which transforms $\mathcal{Q}+q$ into a form $\mathcal{Q}_v$ defined in Section~\ref{sec:null_critical_weights} above Lemma~\ref{lem:F_T^pm_u_in_D_0}. The ground state transform, \cite[Proposition~4.8]{KPP18} states that for all $\varphi \in C_c(X)$, we have
\begin{align*}
	\mathcal{Q}_v(\varphi) = (\mathcal{Q}+q)(v\varphi),
\end{align*}
whenever $v$ is a strictly positive solution of $\mathcal{H}v=0$. 
We denote the operator associated with $\mathcal{Q}_v$ with respect to the measure space $(X,v^2m)$ by $\mathcal{L}_v$ and observe that it acts on 
$\mathcal{F}_v=\frac{1}{v}\mathcal{F}$ as
\begin{align*}
	\mathcal{L}_v f(x) = \frac{1}{(v^2m)(x)} \sum_{y} b(x,y)v(x)v(y) (f(x)-f(y)).
\end{align*}
It is well known that $\mathcal{H}$ is subcritical if and only if $\mathcal{L}_v$ is subcritical, \cite[Corollary~2.5]{KPP20}. We denote the Green's operator for a subcritical $\mathcal{L}_v$ by $G_v$ acting on the corresponding space $\mathcal{G}_v$. Likewise,  we denote
extended space of $\mathcal{Q}_v$ by $\mathcal{D}_{0,v}:=\mathcal{D}_{0}(\mathcal{Q}_v)$ and the space of potentials by $\mathcal{P}_v$.

\begin{lemma} Assume $\mathcal{H}$ is subcritical and that there exists a strictly positive $v\in \mathcal{F}$ such that $\mathcal{H}v=0$. Then, 
	\begin{align*}
		\mathcal{D}_{0,v}=\frac{1}{v}\mathcal{D}_{0}\qquad \mbox{and}\qquad \mathcal{P}_v=\frac{1}{v}\mathcal{P}.
	\end{align*}
\end{lemma}
\begin{proof}
		By the ground state transform we  obtain that  $(\varphi_n)$ converges with respect to $\mathcal{Q}_v$ if and only if $(v\varphi_n)$ converges with respect to $\mathcal{Q}$. This gives the desired relation for the extended spaces. Furthermore, it can be seen by finite approximation of the Green's function, \cite[Theorem~5.16~(b)]{KPP18} that
		\begin{align*}
			G_vk=\frac{1}{v} G vk
		\end{align*}
		for $k \in \mathcal{G}_v=\frac{1}{v}\mathcal{G}$. This gives the desired relation for the space of potentials.
\end{proof}

\begin{proof}[Proof of Theorem~\ref{thm:main_optimality_theorem_Schr}]
	By the lemma above, we have $u/v \in \D_{0,v}\cup \P_v$.
	Since all superharmonic functions in $  \mathcal{D}_{0,v}$  are  necessarily  potentials, cf.\@ \cite[Lemma~8]{HKPII}, and  $\mathcal{L}_v(u/v) = \mathcal{H}u / v\ge0$, we obtain  $u/v \in  \P_v$. Furthermore,
	 $\mathcal{L}_v(u/v) = \mathcal{H}u / v \in v^{-1}\ell^1(X,vm)=\ell^1(X,v^2m)$. Hence, we can apply Theorem~\ref{thm:null_critical_weights_Hardy} to the operator $\mathcal{L}_v$ with vanishing killing term and the function $u/v$ to obtain that $$w = \frac{\mathcal{L}_v (u/v)^{\frac{1}{2}}}{(u/v)^{\frac{1}{2}}}=\frac{\mathcal{H} (uv)^{\frac{1}{2}}}{(uv)^{\frac{1}{2}}}$$ is a null-critical Hardy weight for $\mathcal{Q}_v$ over $(X,v^2m)$, i.e., $\mathcal{Q}_v(\varphi)\ge\sum_X v^2m w \varphi^2$ for $\varphi \in C_c(X)$. By the ground state transform, we have
	\begin{align*}
		(\mathcal{Q}+q)(v\varphi) = \mathcal{Q}_v(\varphi) \geq \sum_X m w (v\varphi)^2
	\end{align*}
	 for all $\varphi \in C_c(X)$.
	This shows that $w $ is a Hardy weight for $\mathcal{Q}+q$. 	Null-criticality of $w$ for $\mathcal{Q}+q$ is equivalent to null-criticality  for $\mathcal{Q}_v$ by \cite[Corollary~2.5]{KPP20}. This finishes the proof.
\end{proof}

\begin{remark}
	Instead of asking for a harmonic function $\mathcal{H}v=0$ in the theorem above, we could also ask for a positive superharmonic function $\mathcal{H}v \geq 0$. In this case, the same proof works if one additionally assumes that
\begin{align*}
	\sum_X m (\mathcal{H}u)v=\sum_X v^2m \mathcal{L}_v (u/v)\neq  \sum_X v^2m (\mathcal{L}_v1) (u/v)=\sum_X m u(\mathcal{H}v).
\end{align*}
\end{remark}
	\medskip

\textbf{Acknowledgement.} The authors acknowledge the financial support of the DFG and the hospitality of the IIAS Jerusalem.


\printbibliography

\end{document}